\documentclass[preprint,review,10pt]{amsart}
\usepackage{amsmath,amssymb,amsfonts,xspace}
\newtheorem{theorem}{Theorem}[section]

\theoremstyle{definition}
\newtheorem{definition}[theorem]{Definition}
\newtheorem{example}[theorem]{Example}

\newtheorem{corollary}[theorem]{Corollary}
\newtheorem{lem}[theorem]{Lemma}

\theoremstyle{remark}

\numberwithin{equation}{section}

\begin{document}
\title{ $\alpha$-prime hyperideals in a multiplicative hyperring }

\author{Mahdi Anbarloei}
\address{Department of Mathematics, Faculty of Sciences,
Imam Khomeini International University, Qazvin, Iran.
}

\email{m.anbarloei@sci.ikiu.ac.ir}


\subjclass[2010]{ 20N20}


\keywords{ $\alpha$-prime hyperideal, $\alpha$-radical, $\alpha$-nilpotent, $\alpha$-nilradical.}

\begin{abstract}
The notion of multiplicative hyperrings is an important class of algebraic hyperstructures which generalize rings where the multiplication is a hyperoperation,
while the addition is an operation . Let $R$ be a commutative  multiplicative hyperring and $\alpha \in End(R)$. A proper hyperideal $I$ of $R$ is called $\alpha$-prime if $x \circ y \subseteq I$ for some $x, y \in R$ then $x \in I$ or $\alpha(y) \in I$.  Indeed, the $\alpha$-prime hyperideals  are  a new generalization of prime hyperideals. In this paper, we aim to  study $\alpha$-prime hyperideals and give the basic properties of this new type of hyperideals. 
\end{abstract}
\maketitle
\section{Introduction}
Algebraic hyperstructures are a suitable generalization of classical algebraic structures. This  theory  has been introduced  by Marty in 1934 during the $8^{th}$ Congress of the Scandinavian Mathematicians \cite{marty}. He defined the hypergroups as a generalization of groups. Afterwards,  many researchers  have been worked on this new ﬁeld of modern algebra and developed it \cite{f1,f2,f3,f4,f5,f6,f7,f8,f9}. In an algebraic hyperstructure, the composition of two elements is a set, while in a classical algebraic structure, the composition of two elements is an element. Similar to hypergroups, hyperrings are algebraic structures more general than rings, subsitutiting both or only one of the binary operations of addition and multiplication by hyperoperations. The hyperrings
were introduced and studied by many authors  \cite{ameri3, ameri4, f12, f13}. Krasner introduced a type of the hyperring  where
addition is a hyperoperation and multiplication is an ordinary binary operation. Such a hyperring is
called a Krasner hyperring \cite{f11}. Mirvakili and Davvaz introduced  $(m,n)$-hyperrings in \cite{f17} and they defined Krasner $(m,n)$-hyperrings as a subclass of $(m,n)$-hyperrings and as a generalization of Krasner hyperrings in \cite{f18}.
The notion of multiplicative hyperrings is an important class of algebraic hyperstructures that
generalize rings, initiated the study  by Rota in 1982 which the multiplication is a hyperoperation,
while the addition is an operation \cite{f14}. There exists a
general type of hyperrings that both the addition and multiplication are hyperoperations \cite{f15}.  Ameri and Kordi have studied Von Neumann regularity in multiplicative hyperrings \cite{ameri5}. Moreover, they introduced the concept of clean multiplicative hyperrings and studied  some topological
concepts to realize clean elements of a multiplicative hyperring by clopen subsets
of its Zariski topology\cite{ameri6}. The
notions such as (weak)zero divisor, (weak)nilpotent and unit in an arbitrary commutative hyperrings were introduced in \cite{ameri2}. Some equivalence relations - called fundamental relations - play important
roles in the the theory of algebraic hyperstructures. The fundamental
relations are one of the most important and interesting concepts in algebraic
hyperstructures that ordinary algebraic structures are derived from algebraic
hyperstructures by them. For more details about hyperrings and fundamental relations on
hyperrings see \cite{x4, f6, x2, x3, x1, marty, x5, f15}.
Prime ideals and primary ideals are two of the most important structures in commutative algebra. The notion of primeness of hyperideal in a multiplicative hyperring was conceptualized by Procesi and rota in \cite{f16}. Dasgupta extended the prime and primary hyperideals in multiplicative hyperrings in \cite{das}. 
Beddani and Messirdi \cite{bed} introduced a generalization of prime ideals called
2-prime ideals and this idea is further generalized by Ulucak and et. al.  \cite{gu}. In \cite{anb}, we investigated $\delta$-primary hyperideals in
a Krasner $(m, n)$-hyperring which unify prime hyperideals and primary hyperideals. $\alpha$-prime ideals in a commutative ring with nonzero identity have been introduced
and studied by Akray and Mohammad-Salih in \cite{akray}.

In this paper we consider the class of multiplicative hyperring as a hyperstructure
$(R,+,\circ)$, where $(R,+)$ is an abelian group, $(R,\circ)$ is a semihypergroup
and the hyperoperation "$\circ$" is distributive with respect to the operation
$" + "$. In this paper we introduce and study the notion of $\alpha$-prime hyperideals of a multiplicative hyperring  which is  a new generalization to  prime hyperideals. Several properties of them are provided. For example we show (Theorem \ref{18}) that  if $R$ is a multiplicative hyperring such that it has zero absorbing property and $\langle 0 \rangle$ is a prime hyperideal of $R$, then \[Nil_\alpha(R) = \bigcap_{I {\text \  is \ \alpha-prime \  of \  R}}I.\]
It is shown (Theorem \ref{17}) that the hyperideal  $I$ of $R$ is $\alpha$-prime if and only if $I/ Ker \alpha$ is prime in $R/Ker \alpha$. We show (Theorem \ref{19}) that thehyperideal  $I$ is $\alpha$-prime if and only if $R/I$ is an $\alpha$-integral hyperdomain. Also, we investigate
the stability of $\alpha$-prime hyperideals in some hyperring-theoric constructions.
\section{Preliminaries}
In this section we give some defnitions and results  which we
need to develop our paper.

A hyperoperation "$\circ $" on nonempty set $G$ is a mapping of $G \times G$ into the family of all nonempty subsets of $G$. Assume that "$\circ $" is a hyperoperation on $G$. Then $(G,\circ)$ is called hypergroupoid. The hyperoperation on $G$ can be extended to  subsets of $G$ as follows. Let $X,Y$ be subsets of $G$ and $g \in G$, then 
\[X \circ Y =\cup_{x \in X, y \in Y}x \circ y, \ \  X \circ g=X \circ \{g\}. \]
A hypergroupoid $(G, \circ)$ is called  a semihypergroup if for all $x,y,z \in G$, $(x \circ y) \circ z=x \circ (y \circ z)$, which is associative. A semihypergroup is said to be a hypergroup if  $g \circ G=G=G \circ g$ for all  $g \in G$. A nonempty subset $H$ of a semihypergroup $(G,\circ)$ is called a
subhypergroup if for all $x \in H$ we have $x \circ H=H=H \circ x$. A commutative hypergroup $(G,\circ)$ is canonical if
\begin{itemize}
\item[\rm{(i)}]~there exists  $e \in G$ with $e \circ x=\{x\}$, for every $x \in G$.
\item[\rm{(ii)}]~for every $x \in G$ there exists a unique $x^{-1} \in G$ with $e \in x \circ x^{-1}$.
\item[\rm{(iii)}]~$x \in y  \circ z $ implies $y \in x \circ z^{-1}$.
\end{itemize}
A nonempty set $R$ with two hyperoperations $"+"$ and $"\circ"$ is called  a
hyperring if $(R,+)$ is a canonical hypergroup , $(R,\circ)$ is a semihypergroup with
$r \circ 0=0 \circ r=0$ for all $r \in R$  and the hyperoperation $"\circ"$ is distributive with respect to $+$, i.e., $x \circ (y+z)=x \circ y+x \circ z$ and  $(x+y) \circ z=x \circ z+y \circ z$ for all $x,y,z \in R$. 

\begin{definition} \cite{f10}
A {\it multiplicative hyperring} is an abelian group $(R,+)$ in which a hyperoperation $\circ $ is defined satisfying the following: 
\begin{itemize}
\item[\rm(i)]~ for all $a, b, c \in R$, we have $a \circ (b \circ c)=(a \circ b) \circ c$;
\item[\rm(ii)]~for all $a, b, c \in R$, we have $a\circ (b+c) \subseteq a\circ b+a\circ c$ and $(b+c)\circ a \subseteq b\circ a+c\circ a$;
\item[\rm(iii)]~for all $a, b \in R$, we have $a\circ (-b) = (-a)\circ b = -(a\circ b)$.
\end{itemize}
\end{definition}
If in (ii) the equality holds then we say that the multiplicative hyperring is strongly distributive.
Recall that $R$ has a zero absorbing property if for all $r \in  R$, $\{0\}=0 \circ r=r \circ 0$.\\
A non empty subset $I$ of a multiplicative hyperring $R$ is a {\it hyperideal} if
\begin{itemize}
\item[\rm(i)]~ If $a, b \in I$, then $a - b \in I$;

\item[\rm(iii)]~ If $x \in I $ and $r \in R$, then $rox \subseteq I$.
\end{itemize}
Let $(\mathbb{Z},+,\cdot)$ be the ring of integers. Corresponding to every subset $A \in P^\star(\mathbb{Z})$ such that $\vert A\vert \geq 2$, there exists a multiplicative hyperring $(\mathbb{Z}_A,+,\circ)$ with $\mathbb{Z}_A=\mathbb{Z}$ and for any $a,b\in \mathbb{Z}_A$, $a \circ b =\{a.r.b\ \vert \ r \in A\}$. 
\begin{definition} \cite{das} 
A proper hyperideal $P$ of $R$ is called a {\it prime hyperideal} if $x\circ y \subseteq P$ for $x,y \in R$ implies that $x \in P$ or $y \in P$. The intersection of all prime hyperideals of $R$ containing $I$ is called the prime radical of $I$, being denoted by $\sqrt{I}$. If the multiplicative hyperring $R$ does not have any prime hyperideal containing $I$, we define $\sqrt{I}=R$. 
\end{definition}
\begin{definition} \cite{ameri}
A proper hyperideal $I$ of $R$ is {\it maximal} in R if for
any hyperideal $J$ of $R$ with $I \subseteq J \subseteq R$ then $J = R$. Also, we say that $R$ is a local multiplicative hyperring, if it has just one maximal hyperideal.
\end{definition}
Let {\bf C} be the class of all finite products of elements of $R$ i.e. ${\bf C} = \{r_1 \circ r_2 \circ . . . \circ r_n \ : \ r_i \in R, n \in \mathbb{N}\} \subseteq P^{\ast }(R)$. A hyperideal $I$ of $R$ is said to be a {\bf C}-hyperideal of $R$ if, for any $A \in {\bf C}, A \cap I \neq \varnothing $ implies $A \subseteq I$.
Let I be a hyperideal of $R$. Then, $D \subseteq \sqrt{I}$ where $D = \{r \in R: r^n \subseteq I \ for \ some \ n \in \mathbb{N}\}$. The equality holds when $I$ is a {\bf C}-hyperideal of $R$(\cite {das}, proposition 3.2). In this paper, we assume that all hyperideals are {\bf C}-hyperideal.
\begin{definition} \cite{das}
A nonzero proper hyperideal $Q$ of $R$ is called a {\it primary hyperideal} if $x\circ y \subseteq Q$ for $x,y \in R$ implies that $x \in Q$ or $y \in \sqrt{Q}$. Since $\sqrt{Q}=P$ is a prime hyperideal of $R$ by Propodition 3.6 in \cite{das}, $Q$ is referred to as a P-primary hyperideal of $R$.
\end{definition}
\begin{definition} 
\cite{ameri} Let $R$ be commutative multiplicative hyperring and $e$ be an identity (i. e., for all $a \in R$, $a \in a\circ e$). An element $x \in R$ in is called {\it unit}, if there exists $y \in R$, such that $e \in x\circ y$.
\end{definition} 
\begin{definition} \cite{ameri}
Let $R$ be a multiplicative hyperring. The element $x \in R$ is
said to be nilpotent if $0 \in x^n$ for some integer $n> 0$.
\end{definition}
\begin{definition} 
A hyperring $R$ is called an {\it integral hyperdomain}, if for all $x, y \in R$,
$0 \in x . y$ implies that $x = 0$ or $y = 0$. 
\end{definition}
\begin{definition} \cite{ameri2}
An element $a \in R$ is said to be zero divizor if there exists $0 \neq b \in R$ such that $0\in a \circ b.$
\end{definition}
\begin{definition} \cite{f10} 
Let $(R_1, +_1, \circ _1)$ and $(R_2, +_2, \circ_2)$ be two multiplicative hyperrings. A mapping $f$ from
$R_1$ into $R_2$ is said to be a {\it good homomorphism} if for all $x,y \in R_1$, $f(x +_1 y) =f(x)+_2 f(y)$ and $f(x\circ_1y) = f(x)\circ_2 f(y)$.\\
Moreover, the kernel  of $f$ is defined by $Ker f=f^{-1}(\langle 0 \rangle)=\{x \in R_1 \ \vert \ f(x) \in \langle 0 \rangle \}.$
\end{definition}
 
\section{ $\alpha$-prime hyperideals  }
\begin{definition}  Let $R$ be a multiplicative hyperring and let $\alpha :R \rightarrow R$ be a fixed good endomorphism. We say that hyperideal $I$ of $R$ is $\alpha$-prime if for
all $x, y \in R$, $x \circ y \subseteq I$  implies $x \in I$ or $\alpha(y) \in I$.
\end{definition}
\begin{example}
Assume that  $(\mathbb{Z},+,\cdot)$ is the ring of integers. Consider the multiplicative hyperring  $(\mathbb{Z},+,\circ)$ in which  $a \circ b =\{2ab,3ab\}$, for all $a,b \in \mathbb{Z}$. Let $\alpha$ is an identity mapping on $(\mathbb{Z},+,\circ)$. Then  $\langle 2 \rangle$ and $\langle 3 \rangle$ are $\alpha$-prime hyperideals of $(\mathbb{Z},+,\circ)$.
\end{example}
\begin{example}
Let $(\mathbb{Z},+,\cdot)$ is the ring of integers. Consider the multiplicative hyperring $(\mathbb{Z},+,\langle 2 \rangle)$. Define the mapping $\alpha: \mathbb{Z} \longrightarrow \mathbb{Z}$ by $\alpha(x)=3x$.  By Theorem 4.2.3 in \cite{f10}, $\alpha$ is a homomorphism of the multiplicative hyperring $(\mathbb{Z},+,\langle 2 \rangle)$. In the multiplicative hyperring, the hyperideal $\langle 2 \rangle$ is  $\alpha$-prime.
\end{example}
\begin{theorem}
If $I$ is an $\alpha$-prime hyperideal of $R$, then $\alpha(I) \subseteq I$.
\end{theorem}
Let $x \in I$. Since $I$ is a ${\bf C}$-hyperideal of $R$ then  $x \in 1 \circ x \subseteq I$. Since $I$ is an $\alpha$-prime hyperideal of $R$ and  $1 \notin I$ then $\alpha(x) \in I$ for all $x \in I$. Thus we have $\alpha(I) \subseteq I$.
\begin{lem}
If $I$ is an $\alpha$-prime hyperideal of $R$, then $\sqrt{I}$ is an $\alpha$-prime hyperideal of $R$.
\end{lem}
\begin{proof}
Assume that $a \circ b \subseteq \sqrt{I}$ for $a,b \in R$. This means $(a \circ b)^n=a^n \circ b^n \subseteq I$. Since $I$ is an $\alpha$-prime hyperideal of $R$, then we get $a^n \subseteq I$ or $\alpha(b^n)=(\alpha(b))^n \subseteq I$. This implies that $a \in \sqrt{I}$ or $\alpha(b) \subseteq \sqrt{I}$. Thus $\sqrt {I}$ is an $\alpha$-prime hyperideal of $R$. 
\end{proof}
\begin{theorem}
Let $I$ is an $\alpha$-prime hyperideal of $R$. Then $E=\{s \in R \ \vert \ \alpha(s) \in I \} $ is an $\alpha$-prime hyperideal of $R$ containing $I$. 
\end{theorem}
\begin{proof}
It is clear that $E$ is a hyperideal of $R$ containing $I$. Let $a \circ b \subseteq E$ for some $a,b \in R$. This means $\alpha(a \circ b)=\alpha(a) \circ \alpha(b) \subseteq I$. Since $I$ is an $\alpha$-prime hyperideal of $R$ then $\alpha(a)\in I$ or $\alpha( \alpha(b)) \in I$. This implies that $a \in E$ or $\alpha(b) \in E$. Thus $E=\{s \in R \ \vert \ \alpha(s) \in I \} $ is an $\alpha$-prime hyperideal of $R$. 
\end{proof}
\begin{lem}
Let $I$ is an $\alpha$-prime hyperideal of $R$ and let $I$ be maximal with respect to the fact that $s \in I$ implies $\alpha(s) \in I$. Then $I$ is prime.
\end{lem}
\begin{proof}
Assume that $I$ is not prime hyperideal of $R$. Then we get  $x,y \in R$ with $x \circ y \subseteq I$ such that neither $x \in I$ nor $y \in I$.  Let $a=p+t \in  I+\langle x \rangle$ for $p \in I$ and $t \in \langle x \rangle$. Then there exists $r \in R$ with  $t \in r \circ x $ such that $a \in p+r \circ x$. Therefore we have $a \circ y \subseteq (p+r \circ x) \circ y \subseteq p \circ y +r \circ x \circ y \subseteq I$. Since $I$ is an $\alpha$-prime hyperideal of $R$ and $y \notin I$ then $\alpha(a) \in I \subseteq I+\langle x \rangle$. By hypothesis, $I$ is maximal. Hence $I=I+\langle x \rangle$  which implies $x \in I$. This a contradiction. Thus $I$ is a prime hyperideal of $R$.  
\end{proof}
\begin{theorem}
The hyperideal $I$ of $R$ is $\alpha$-prime  if and only if for any hyperideals $I_1, I_2$ of $R$, $\hspace{0.25cm} I_1 \circ I_2\subseteq I  \Longrightarrow I_1 \subseteq I \hspace{0.15cm}\text {or} \hspace{0.15cm}   \alpha(I_2) \subseteq I$.
\end{theorem}
\begin{proof}
Let $I$ is an $\alpha$-prime hyperideal of $R$. Assume that $I_1 \circ I_2 \subseteq I$ for some hyperideals $I_1, I_2$ of $R$ such that $I_1 \nsubseteq I$. Then there exists $x \in I_1$ but $x \notin I$. Clearly, $x \circ y \subseteq I_1 \circ I_2 \subseteq I$ for all $y \in I_2$. Since the hyperideal $I$  is $\alpha$-prime and $x \notin I$ then $\alpha(y) \in I$ for all $y \in I_2$ which means $\alpha(I_2) \subseteq I$. Conversely, suppose that $x \circ y \subseteq I$ for some $x,y \in R$. Hence $\langle x \rangle \circ \langle y \rangle \subseteq \langle x \circ y \rangle \subseteq I$. By hypothesis, we get $\langle x \rangle \subseteq I$ or $\alpha(\langle y \rangle) \subseteq I$ which implies $x \in I$ or $\alpha(y) \in I$. Thus $I$ is an $\alpha$-prime hyperideal of $R$.
\end{proof}
\begin{theorem}
Let $I$ is an $\alpha$-prime hyperideal of $R$ and $S$ is a subset of $R$. Then $(I:S)$ is an $\alpha$-prime hyperideal of $R$.
\end{theorem}
\begin{proof}
Let $x \circ y \subseteq (I:S)$ for some $x,y \in R$. It is easy to see $(I:x\circ S)=(I:S) \cup (I:x)$. Since $y \in (I:x \circ S)$ then we have $y \in (I:S)$ or $y \circ x \subseteq I$. Hence we get $y \in (I:S)$ or $y \in I$ or $\alpha(x) \in I$, since $I$ is an $\alpha$-prime hyperideal of $R$. This implies that $y \in (I:S)$ or $\alpha(x) \in (I:S)$. Thus $(I:S)$ is an $\alpha$-prime hyperideal of $R$.
\end{proof}
\begin{lem} \label{11}
Let $I$ is an $\alpha$-prime hyperideal of $R$. If $x^n \subseteq I$ for some $x \in R$, then $\alpha(x) \in I$.
\end{lem}
\begin{proof}
Let $x^n \subseteq I$ for some $x \in R$. Assume that $t \in x^{n-1}$. Then $t \circ x \subseteq I$. Since $I$ is an $\alpha$-prime hyperideal of $R$ then we have $t \in I$ or $\alpha(x) \in I$. Let $t \in I$. Since $t \in x^{n-1}$ and $I$ is a {\bf C}-hyperideal of $R$ then we get $x^{n-1} \subseteq I$. Continuing this process, we obtain $x \in I$ or $\alpha(x) \in I$. If $x \in I$, then $x \in 1 \circ x \subseteq I$, that is, $\alpha(x) \in I$. 
\end{proof}
\begin{theorem}
Let $I$ is an $\alpha$-prime hyperideal of $R$. If $(\alpha(y))^n \subseteq I$ for some $y \in R$, then $\alpha^2(y) \in I$.
\end{theorem}
\begin{proof}
Let $I$ is an $\alpha$-prime hyperideal of $R$ such that $(\alpha(y))^n \subseteq I$ for some $y \in R$. Assume that $x=\alpha(y)$. Now the claim follows by Lemma \ref{11}.
\end{proof}
\begin{definition}
Let $R$ be a multiplicative hyperring. An element $x$ of $R$ is said to be $\alpha$-nilpotent, if $0 \in \alpha(x^n)$, for some integer $n >0$. We denote the set of $\alpha$-nilpotent elements of $R$ by $Nil_{\alpha}(R)$ and call it the $\alpha$-nilradical of $R$.
\end{definition}
\begin{theorem}
The set $Nil_{\alpha}(R)$ of all $\alpha$-nilpotent elements of $R$ with scalar identity 1, is a hyperideal. 
\end{theorem}
\begin{proof}
Suppose that $x \in Nil_{\alpha}(R)$. Then $0 \in \alpha(x^n)$ for some integer $n>0$. Hence for all $r \in R$, we get $0 \in \alpha(r^n) \circ \alpha(x^n)=\alpha((r\circ x)^n)$ which implies $r \circ x \in  Nil_{\alpha}(R)$. Now, Suppose that $x,y \in Nil_{\alpha}(R)$, then there exist $n,m \in \mathbb{N}$ such that $0 \in \alpha(x^n)$ and $0 \in \alpha(y^n)$. Thus we have $0 \in \alpha ((x-y)^{n+m})$. Therefore $x-y \in Nil_{\alpha}(R)$. Consequently, $Nil_{\alpha}(R)$ is a hyperideal.
\end{proof}
\begin{theorem} \label{12}
Let  $R_1$ and $R_2$ be two multiplicative hyperrings and  $f:R_1 \rightarrow R_2$  a good homomorphism. If $\alpha \in End(R_1) \cap End(R_2)$ such that $\alpha(f(r))=f(\alpha(r))$ for every $r \in R_1$, then $f^{-1}(I_2)$ is an $\alpha$-prime hyperideal of $R_1$ for some $\alpha$-prime hyperideal $I_2$ of $R_2$.
\end{theorem}
\begin{proof}
Assume that the hyperideal $I_2$ of $R_2$ is  $\alpha$-prime. Let $x \circ_1 y \subseteq f^{-1}(I_2)$ for some $x,y \in R_1$. Then $f(x \circ_1 y)=f(x) \circ_2 f(y) \subseteq I_2$. Since $I_2$ is an $\alpha$-prime hyperideal of $R_2$ then $f(x) \in I_2$ which implies $x \in f^{-1}(I_2)$ or $\alpha(f(y))=f(\alpha(y)) \in I_2$ which implies $\alpha(y) \in f^{-1}(I_2)$. Thus $f^{-1}(I_2)$ is an $\alpha$-prime hyperideal of $R_1$.
\end{proof}
\begin{lem} \label{13}
Let $\alpha \in End(R)$. Then $Ker \alpha \subseteq \bigcap_{I {\text \  is \ \alpha-prime \  of \  R}}I$.
\end{lem}
\begin{proof}
Let $r \in Ker \alpha$. Therefore $\alpha(r) \in \langle 0 \rangle$. This means $\alpha(r)$ is in every $\alpha$-prime hyperideal $I$ of $R$. Now by using Theorem \ref{12}, the claim can be proved.
\end{proof}
\begin{lem} \label{14}
Let $\langle 0 \rangle$ be a prime hyperideal of $R$ and let $\alpha \in End (R)$. Then $Ker \alpha$ is a prime hyprideal of $R$.
\end{lem}
\begin{proof}
Let $a \circ b \subseteq Ker \alpha$ for some $a,b \in R$. This implies that $\alpha(a \circ b)=\alpha(a) \circ \alpha(b) \subseteq \langle 0 \rangle$. Since $\langle 0 \rangle$ is a a prime hyperideal of $R$, then we have $\alpha(a) \in \langle 0 \rangle$ or $\alpha(b) \in \langle 0 \rangle$. This implies that $a \in Ker \alpha$ or $a \in Ker \alpha$. Thus $Ker \alpha$ is a prime hyprideal of $R$.
\end{proof}
\begin{theorem} \label{18}
Let $R$ be a multiplicative hyperring such that it has zero absorbing property. If  $\langle 0 \rangle$ be a prime hyperideal of $R$, then \[Nil_\alpha(R) = \bigcap_{I {\text \  is \ \alpha-prime \  of \  R}}I.\]
\end{theorem}
\begin{proof}
Let $r \in Nil_\alpha(R)$. Then we have $0 \in \alpha(r^n)$ which means $r^n \subseteq Ker \alpha$. Thus we get $r \in Ker \alpha$, by Lemma \ref{14}. Hence we conclude that  $r \in \bigcap_{I {\text \  is \ \alpha-prime \  of \  R}}I$. Then $Nil_\alpha(R)$ is in the intersection of
all $\alpha$-prime hyperideals of $R$.\\ Now, assume that $r \in \bigcap_{I {\text \  is \ \alpha-prime \  of \  R}}I$ but $r \notin Nil_\alpha(R)$. Consider the set \[\Sigma=\{J \ \vert \ \text{J is a hyperideal of R and for all n>0, $\alpha(r^n) \nsubseteq$ J }\}.\] Since $0 \in \Sigma$ then $\Sigma \neq \varnothing$. Order $\Sigma$ by inclusion. Assume that $\{J_i\}_{i \in \Delta}$ is a chain of hyperideals in $\Sigma$, then for each pair of indices $t,s$ we have either $J_{t} \subseteq J_{s}$ or $J_{s} \subseteq J_{t}$. Let $J=\bigcup_{i \in \Delta} J_i$. Clearly, $J$ is a hyperideal and is an upper bound of the chain. Thus by Zorn's lemma $\Sigma$ has a
maximal element. Let $P$ is a maximal element of $\Sigma$. Assume that $\alpha(x) \notin P$ and $\alpha(y) \notin P$ for some $x,y \in R$. Hence  $P+\langle \alpha(x) \rangle$ and $P+\langle \alpha(y) \rangle$ are not in $\Sigma$. Then we have $\alpha(r^m) \subseteq P+\langle \alpha(x) \rangle$ and $\alpha(r^n) \subseteq P+\langle \alpha(y) \rangle$ for some integers $m,n >0$. Therefore $\alpha(r^{m+n}) \subseteq P+\langle \alpha(x \circ y) \rangle$. This means $P+\langle \alpha(x \circ y) \rangle \notin \Sigma$ which implies $\alpha(x \circ y)=\alpha(x) \circ \alpha (y) \nsubseteq P$. Thus by Lemma \ref{11} $x \circ y \nsubseteq P$. Since $P$ is an $\alpha$-prime hyperideal of $R$ and $\alpha(r^n) \nsubseteq P$ then $r \notin P$ which  is a contradiction. Therefore  $r \in Nil_\alpha(R)$ and the proof is completed.
\end{proof}
\begin{definition}
Let $J$ be a hyperideal of $R$ such that $R$  has zero absorbing property. The $\alpha$-radical of $J$ is defined by
\[\sqrt[\alpha]{J}=\{r \in R \ \vert \ \alpha(r^n) \subseteq J \ \text {for some} \  n \in \mathbb{N}\}\]
\end{definition}
\begin{theorem}
$Nil_{\alpha}(R) \subseteq \sqrt[\alpha]{\langle 0 \rangle}$
\end{theorem}
 \begin{proof}
 Let $r \in Nil_{\alpha}(R)$. Then there exists some $n \in \mathbb{N}$ such that $0 \in \alpha(r^n)$. Since $\langle 0 \rangle$ is a {\bf C}-hyperideal and $0 \in \langle 0 \rangle$ then $\alpha(r^n) \subseteq \langle 0 \rangle$. Therefore $r \in \sqrt[\alpha]{\langle 0 \rangle}$. Thus $Nil_{\alpha}(R) \subseteq \sqrt[\alpha]{\langle 0 \rangle}$.
 \end{proof}
If $R$ is a multiplicative hyperring such that it has zero absorbing property, then we have $Nil_{\alpha}(R) = \sqrt[\alpha]{\langle 0 \rangle}$
\begin{theorem} \label{15}
Let $A,B$ be two hyperideals of $R$. Then we have the following statements:
\begin{itemize}
\item[\rm(i)]~ If $A \subseteq B$, then $\sqrt[\alpha]{A} \subseteq \sqrt[\alpha]{B}$.
\item[\rm(ii)]~$\sqrt[\alpha]{A + B} \subseteq  \sqrt[\alpha]{\sqrt[\alpha]{A} + \sqrt[\alpha]{ B}}$
\item[\rm(iii)]~$\sqrt[\alpha]{A \circ B}=\sqrt[\alpha]{A \cap B}=\sqrt[\alpha]{A} \cap \sqrt[\alpha]{B}$.
\end{itemize}
\end{theorem}
\begin{proof}
\begin{itemize}
\item[\rm i.]~Straightforward.
\item[\rm ii.]~ Since $A \subseteq \sqrt[\alpha]{A}$ and $B \subseteq \sqrt[\alpha]{B}$ then we have $A+B \subseteq \sqrt[\alpha]{A}+\sqrt[\alpha]{B}$. Thus we get $\sqrt[\alpha]{A + B} \subseteq  \sqrt[\alpha]{\sqrt[\alpha]{A} + \sqrt[\alpha]{ B}}$, by (i).
\item[\rm iii.]~ Here $A \circ B \subseteq A \cap B$. Then $\sqrt[\alpha]{ A \circ B}  \subseteq \sqrt[\alpha]{ A \cap B}$. Now, let $r \in \sqrt[\alpha]{ A \cap B}$. So $\alpha(r^n) \subseteq A \cap B$ for some $n \in \mathbb{N}$. Hence we have $\alpha(r^n) \circ \alpha(r^n)=\alpha(r^{2n}) \subseteq A \circ B$ which means $r \in \sqrt[\alpha]{ A \circ  B}$. Finally, let $r \in \sqrt[\alpha]{A} \cap \sqrt[\alpha]{B}$. This means $\alpha(r^s) \subseteq A$ and $\alpha(r^t) \subseteq B$ for some $t, s \in \mathbb{N}$. Then we have $\alpha(r^m) \subseteq A \cap B$ for $m=max\{t,s\}$. This implies that $r \in \sqrt[\alpha]{A \cap B}$. For the reverse inclusion, since $A \cap B \subseteq A$ and $A \cap B \subseteq A$ then we get $\sqrt[\alpha]{A \cap B} \subseteq \sqrt[\alpha]{A}  \cap \sqrt[\alpha]{A}$.
\end{itemize}
\end{proof}
\begin{theorem}
Let $A$ be a hyperideal of $R$. 
\begin{itemize}
\item[\rm(1)]~If $\alpha(1)=1$, then $\sqrt[\alpha]{A}=R$ if and only if $I=R$.
\item[\rm(2)]~ If the hyperideal $A$ is $\alpha$-prime, then $\sqrt[\alpha]{A^n}=\sqrt[\alpha]{A}$, for all $n \in \mathbb{N}$.
\end{itemize}
\end{theorem}
\begin{proof}
\begin{itemize}
\item[\rm(1)]~ Let $\sqrt[\alpha]{A} = R$. This means $1 \in \sqrt[\alpha]{A}$. Hence $\alpha(1^n) \subseteq A$ which implis $\alpha(1)=1 \in A$. Thus $A=R$
\item[\rm(2)]~ Let the hyperideal $A$ be $\alpha$-prime. Then $\sqrt[\alpha]{A^n}=\sqrt[\alpha]{A} \cap ... \cap \sqrt[\alpha]{A}$ for all $n \in \mathbb{N}$, by Theorem \ref{15} (iii). Thus $\sqrt[\alpha]{A^n}=\sqrt[\alpha]{A}$.
\end{itemize}
\end{proof}
\begin{theorem}
Let  $R_1$ and $R_2$ be two multiplicative hyperrings and  $f:R_1 \rightarrow R_2$  a good homomorphism such that $I_1$ and $I_2$ are hyperideals of $R_1$ and $R_2$, respectively. Assume that  $\alpha \in End(R_1) \cap End(R_2)$ such that $\alpha(f(r))=f(\alpha(r))$ for every $r \in R_1$. Then 
\begin{itemize}
\item[\rm(1)]~$ f(\sqrt[\alpha]{I_1}) \subseteq \sqrt[\alpha]{f(I_1)}$.
\item[\rm(2)]~$ \sqrt[\alpha]{f^{-1}(I_2)} \subseteq f^{-1}(\sqrt[\alpha]{I_2})$.
\item[\rm(3)]~If $f$ is an isomorphism, then  $f(\sqrt[\alpha]{I_1}) = \sqrt[\alpha]{f(I_1)}$
\end{itemize}
\end{theorem}
\begin{proof}
\begin{itemize}
\item[\rm(1)] Let $y \in f(\sqrt[\alpha]{I_1})$. Then there exists some $x \in \sqrt[\alpha]{I_1}$ such that $f(x)=y$. Hence we have $\alpha(x^n) \subseteq I_1$ for some $n \in \mathbb{N}$. Therefore $\alpha(y^n)=\alpha(f(x)^n)=\alpha(f(x^n))$. Since $\alpha$ commutes with $f$, then we get $\alpha(f(x^n))=f(\alpha(x^n)) \subseteq f(I_1)$. Thus $y \in \sqrt[\alpha]{f(I_1)}$ which means  $ f(\sqrt[\alpha]{I_1}) \subseteq \sqrt[\alpha]{f(I_1)}$.
\item[\rm(2)] Let $x \in \sqrt[\alpha]{f^{-1}(I_2)}$. Then we get $\alpha(x^n) \subseteq f^{-1}(I_2)$ for some $n \in \mathbb{N}$. Therefore $f(\alpha(x^n)) \subseteq I_2$ which implies $\alpha(f(x)^n) \subseteq I_2$. So $x \in f^{-1}(\sqrt[\alpha]{I_2})$. 
\item[\rm(3)] Let $f$ is an isomorphism. The  claim follows by (1).
\end{itemize}
\end{proof}
\begin{theorem}
Let $\alpha \in End(R)$. Assume that $I$ is a hyperideal of $R$ such that for all $a,b \in R$, $a\circ b \subseteq I$ implies $a \in I$ or $\alpha(b^n) \subseteq I$ for some $n \in \mathbb{N}$. Then $\sqrt[\alpha]{I}$ is an $\alpha$-prime hyperideal of $R$.  
\end{theorem}
\begin{proof}
Let $x \circ y \subseteq \sqrt[\alpha]{I}$ for some $x,y \in R$. This means $\alpha((x \circ y)^n)=\alpha(x^n) \circ \alpha(y^n) \subseteq I$ for some $n \in \mathbb{N}$. Let $t \in x^n$ and $s \in y^n$ for some $t,s \in R$. Therefore $\alpha(t)\circ \alpha(s) \subseteq \alpha(x^n) \circ \alpha(y^n) \subseteq I$. By assumption, we get $\alpha(t) \in I$ or $\alpha(\alpha(s)^m) \subseteq I$ for some $m \in \mathbb{N}$. Since $I$ is a {\bf C}-hyperideal of $R$ and $\alpha(x^n) \cap I \neq \varnothing$ or $\alpha(\alpha(y)^{nm}) \cap I \neq \varnothing$, then we have $\alpha(x^n) \subseteq  I$ or $\alpha(\alpha(y)^{nm}) \subseteq I $. This implies that $x \in \sqrt[\alpha]{I}$ or $\alpha(y) \in \sqrt[\alpha]{I}$ which means $\sqrt[\alpha]{I}$ is an $\alpha$-prime hyperideal of $R$.
\end{proof}

\begin{theorem} \label{17}
Let $I$ be a hyperideal of $R$. Then $I$ is $\alpha$-prime if and only if every zero divizor of $R/I$ is in $Ker \alpha$.
\end{theorem}
\begin{proof}
Let $I$ be an $\alpha$-prime  hyperideal of $R$. Let $0_{R/I} \neq y+I$ be a zero divizor of $R/P$. Then there exists $0_{R/I} \neq x+I$ such that $0_{R/I} \in (x+I)(y+I)=x \circ y +I$. This implies that $x \circ y \subseteq I$. Since the hyperideal $I$ of $R$ is $\alpha$-prime, then we get $x \in I$ or $\alpha(y) \in I$. Since  $0_{R/I} \neq x+I$, then we have $\alpha(y) \in I$ and so $\alpha(y+I) \subseteq I$which means $y+I$ is in $Ker \alpha$. Conversely, Let $x \circ y \subseteq I$ such that $x,y \notin I$ for some $x,y \in R$. Then $I \in x \circ y + I=(x+I)(y+I)$. Then $y+I$ is a zero divizor of $R/P$. By hypothesis, $\alpha(y+I) \subseteq I$ which means $\alpha(y) \in I$,  as claimed.
\end{proof}
The following lemma is needed in the proof of our next result.
\begin{lem} \label{16}
Let $I$ be a hyperideal of $R$. Then $I$ is prime if and only if $R/I$ has no zero divisors. 
\end{lem}
\begin{proof}
Let $I$ be a prime  hyperideal of $R$. Let $I \neq x+I$ is a zero divisor of $R/I$. Then there exists $I \neq y+I$ such that $I \in (x+I)(y+I)=x \circ y +I$ which means $x \circ y \subseteq I$. Since $I$ is a prime hyperideal of $R$, then we get $x \in I$ or $y \in I$, contradicion. Conversely, let for some $x,y \in R$, $x \circ y \subseteq I$. Then $I \in x \circ y +I=(x+I)(y+I)$. Since $R/I$ has no zero divisors, then we have $I=x+I$ or $I=y+I$ which means $x \in I$ or $y \in I$. 
\end{proof}
\begin{theorem}
Let $I$ be a hyperideal of $R$. Then $I$ is $\alpha$-prime if and only if $I/ Ker \alpha$ is prime in $R/Ker \alpha$.
\end{theorem}
\begin{proof}
By Lemma \ref{13}, we conclude that $R/I \cong \frac{R/Ker \alpha}{I/Ker \alpha}$. Now, the claim follows by Lemma \ref{16} and Theorem \ref{17}.
\end{proof}
\begin{definition}
A hyperring $R$ is called an $\alpha$-integral hyperdomain, if for all $x,y \in R$, $0 \in x \circ y$ implies that $x=0$ or $\alpha(y)=0$.
\end{definition}
\begin{theorem} \label{19}
Let $I$ be a hyperideal of $R$. Then $I$ is $\alpha$-prime if and only if $R/I$ is an $\alpha$-integral hyperdomain.
\end{theorem}
\begin{proof}
Let the hyperideal $I$ of $R$ be $\alpha$-prime. Assume that $I \in (x+I)(y+I)=x \circ y+I$ for some $x,y \in R$. Then $x \circ y \subseteq I$. Therefore we get $x \in I$ or $\alpha(y) \in  I$, since $I$ is a $\alpha$-prime hyperideal of $R$. Hence we conclude that $x+I=I$ or $\alpha(y)+I=I$ which implies $x+I=I$ or $\alpha(y+I)=I$. Consequently, $R/I$ is an $\alpha$-integral hyperdomain. Conversely, Let $R/I$ be an $\alpha$-integral hyperdomain. Suppose that $x \circ y \subseteq I$ for some $x,y \in R$. Then $I \in x \circ y+I$. This means $I \in (x+I)(y+I)$. Thus we have $I=x+I$ or $I=\alpha(y+I)$, since $R/I$ is an $\alpha$-integral hyperdomain. This means $x \in I$ or  $\alpha(y) \in I$. Thus the hyperideal $I$ of $R$ is $\alpha$-prime.
\end{proof}
\begin{theorem} \label{20}
Let  $R_1$ and $R_2$ be two multiplicative hyperrings and  $f:R_1 \rightarrow R_2$  a good epimorphism and  $\alpha \in End(R_1) \cap End(R_2)$ such that $\alpha(f(r))=f(\alpha(r))$ for every $r \in R_1$. Let $I_1$ be a hyperideal of $R_1$ with $Ker \alpha \subseteq I_1$. Then the hyperideal $I_1$ is $\alpha$-prime if and only if the hyperideal $f(I_1)$ of $R_2$ is  $\alpha$-prime.
\end{theorem}
\begin{proof}
 Let $a_2 \circ b_2 \subseteq f(I_1)$ for some $a_2,b_2 \in R_2$. Then for some $a_1,b_1 \in R_1$ we have $f(a_1)=a_2$ and $f(b_1)=b_2$. So $f(a_1) \circ f(b_1)=f(a_1 \circ b_1) \subseteq f(I_1)$. Now, take any $u \in a_1 \circ b_1$. Then $f(u) \in f(a_1 \circ b_1) \subseteq f(I_1)$ and so there exists $w \in I_1$ such that $f(u)=f(w)$. This means that $f(u-w)=0$, that is, $u-w \in Ker f \subseteq I_1$ and then $u \in I_1$. Since $I_1$ is a {\bf C}-hyperideal of $R_1$, then we get $a_1 \circ b_1 \subseteq I_1$. Since $I_1$ is an $\alpha$-prime hyperideal of $R_1$, then we obtain $a_1 \in I_1$ or $\alpha(b_1) \in I_1$. This implies that $f(a_1)=a_2 \in f(I_1)$ or $\alpha(b_2)=\alpha(f(b_1))=f(\alpha(b_1)) \in f((I_1)$. Thus $f(I_1)$ is an $\alpha$-prime hyperideal of $R_2$. The converse part is follows by \ref{12}.
\end{proof}
In view of Theorem \ref{20}, we have the following result.
\begin{corollary}
Let $I$ and $J$ be two hyperideals of $R$ with $J \subseteq I$. Assume that $\alpha \in End(R)$ and $\alpha^\star$ is the induced mapping on $R/J$ from $\alpha$. Then $I$ is an $\alpha$-prime hyperideal of $R$ if and only if $I/J$ is an $\alpha^\star$-prime hyperideal of $R/J$.
\end{corollary} 
Let $(R_1,+_1,\circ_1)$ and $(R_2,+_2,\circ_2)$ be two multiplicative hyperrings with non zero identity. \cite{ul} Recall $(R_1 \times R_2, +,\circ)$ is a multiplicative hyperring with the operation $+$ and the hyperoperation $\circ$ are defined respectively as

$(x_1,x_2)+(y_1,y_2)=(x_1+_1y_1,x_2+_2y_2)$ 
and

$(x_1,x_2) \circ (y_1,y_2)=\{(x,y) \in R_1 \times R_2 \ \vert \ x \in x_1 \circ_1 y_1, y \in x_2 \circ_2 y_2\}$. \\
Assume that  $\alpha_1 \in End(R_1)$ and $\alpha_2 \in End(R_2)$. We  define the map $\bar{\alpha}: R_1 \times R_2 \longrightarrow R_1 \times R_2$ by $\bar{\alpha}(r_1,r_2)=(\alpha_1(r_1), \alpha_2(r_1))$. It  is easy to see that $\bar{\alpha} \in End(R_1 \times R_2)$.

\begin{theorem} \label{times}
Let $(R_1, +_1,\circ _1)$ and $(R_2,+_2,\circ_2)$ be two multiplicative hyperrings with non zero identity such that $\alpha_1 \in End(R_1)$ and $\alpha_2 \in End(R_2)$. Let $I_1$ be a hyperideal of $R_1$. Then  $I_1$ is an $\alpha_1$-prime hyperidea of  $R_1$   if and only if $I_1 \times R_2$ is an $\bar{\alpha}$-prime hyperideal of $R_1 \times R_2$.
\end{theorem}
\begin{proof}
($\Longrightarrow$) Let $(x_1,x_2) \circ (y_1,y_2) \subseteq I_1 \times R_2$ for some $(x_1,x_2), (y_1,y_2) \in R_1 \times R_2$. This means $x_1 \circ_1 y_1 \subseteq I_1$. Since $I_1$ is a $\alpha_1$-prime hyperideal of $R_1$, then we get $x_1 \in I_1$ or $\alpha_1(y_1) \in I_1$. This implies that $(x_1,x_2) \in I_1\times R_2$ or $\bar{\alpha}(y_1,y_2)=(\alpha_1(y_1),\alpha_2(y_2) \in I_1 \times R_2$. Consequently, $I_1 \times R_2$ is an $\bar{\alpha}$-prime hyperideal of $R_1 \times R_2$.

($\Longleftarrow$) Assume on the contrary that $I_1$ is not a $\alpha_1$-prime hyperideal of $R_1$. So $x_1 \circ_1 y_1 \subseteq I_1$ with $x_1, y_1 \in R_1$ implies that $x_1 \notin I_1$ and $\alpha_1(y_1) \notin I_1$. It is clear that $(x_1,1_{R_2})\circ (y_1, 1_{R_2}) \subseteq I_1 \times R_2$. Since $I_1 \times R_2$ is an $\bar{\alpha}$-prime hyperideal of $R_1 \times R_2$, then we have $(x_1,1_{R_2}) \in I_1 \times R_2$ or $\bar{\alpha}(y_1,1_{R_2}) \in I_1 \times R_2$ which means $(x_1,1_{R_2}) \in I_1 \times R_2$ or $(\alpha_1(y_1), \alpha_2(1_{R_2})\in I_1 \times R_2$. Hence we get $x_1 \in I_1$ or $\alpha_1(y_1) \in I_1$ which is a contradiction. Thus, $I_1$ is an $\alpha_1$-prime hyperideal of $R_1$. 
\end{proof}
\begin{theorem}
Let $(R_1, +_1,\circ 1)$ and $(R_2,+_2,\circ_2)$ be two multiplicative hyperrings with non zero identity such that $\alpha_1 \in End(R_1)$ and $\alpha_2 \in End(R_2)$. Let $I_1$ and $I_2$ be some hyperideals of $R_1$ and $R_2$, respectively. Then the
following statements are equivalent:
\begin{itemize}
\item[\rm{(1)}]~ $I_1 \times I_2$ is an $\bar{\alpha}$-prime hyperideal of $R_1 \times R_2$.
\item[\rm{(2)}]~ $I_1=R_1$ and $I_2$ is an $\alpha_2$-prime hyperideal of $R_2$ or $I_2=R_2$ and $I_1$ is an $\alpha_1$-prime hyperideal of $R_1$.
\end{itemize}
\end{theorem}
\begin{proof}
(1) $\Longrightarrow$ (2) Assume that $I_1=R_1$. Then $I_2$ is a $\alpha_2$-primary hyperideal of $R_2$, by Theorem \ref{times}.

(2) $\Longrightarrow$ (1) This can be proved by using Theorem \ref{times}.
\end{proof}
\begin{example} 
Suppose that $(\mathbb{Z},+,.)$ is the ring of integers. Then $(\mathbb{Z},+,\circ_1)$ is a multiplicative hyperring with
a hyperoperation $a \circ_1 b =\{ab,7ab\}$. Also, $(\mathbb{Z},+,\circ_2)$ is a multiplicative hyperring with
a hyperoperation $a \circ_2 b =\{ab,5ab\}$. Note that $(\mathbb{Z} \times \mathbb{Z},+,\circ)$ is a multiplicative hyperring with a hyperoperation $(a,b) \circ (c,d)=\{(x,y) \in \mathbb{Z} \times \mathbb{Z} \ \vert \ x \in a \circ_1 c , y \in b \circ_2
d \}$. Let $\alpha_1$ and $\alpha_2$ are the identity maps on $(\mathbb{Z},+,\circ_1)$ and  $(\mathbb{Z},+,\circ_2)$, respectively. Clearly, $7\mathbb{Z}=\{7t \ \vert \ t \in \mathbb{Z}\}$ and $5\mathbb{Z}=\{5t \ \vert \ t \in \mathbb{Z}\}$ are  $\alpha_1$-prime and  $\alpha_2$-prime of $(\mathbb{Z},+,\circ_1)$ and $(\mathbb{Z},+,\circ_2)$, respectively. Since $(5,0) \circ (0,7) \subseteq 7\mathbb{Z} \times 5 \mathbb{Z}
$ but $(5,0),(0,7) \notin 7\mathbb{Z} \times 5\mathbb{Z}$ and $\bar{\alpha}(5,0),\bar{\alpha}(0,7) \notin 7\mathbb{Z} \times 5\mathbb{Z}$, then $7\mathbb{Z} \times 5\mathbb{Z}$ is not a $\bar{\alpha}$-prime hyperideal of $\mathbb{Z} \times \mathbb{Z}$.
\end{example}


\end{document}